\documentclass[12pt]{amsart}
\usepackage{amscd,verbatim}
\usepackage{times}
\usepackage[%dvipdfmx,
colorlinks,linkcolor=blue,citecolor=blue,urlcolor=red]{hyperref}
\usepackage[all]{xy}

\newcommand{\sF}{\mathcal{F}}

\newcommand{\G}{\mathbb{G}}

\newcommand{\Q}{\mathbf{Q}}
\newcommand{\Z}{\mathbf{Z}}
\renewcommand{\epsilon}{\varepsilon}

\newcommand{\Spec}{\operatorname{Spec}}

\newcommand{\uHom}{\operatorname{\underline{Hom}}}

\newcommand{\et}{{\operatorname{\acute{e}t}}}
\newcommand{\Nis}{{\operatorname{Nis}}}
\newcommand{\eff}{{\operatorname{eff}}}

\renewcommand{\o}{{\operatorname{o}}}
\newcommand{\DM}{\operatorname{DM}}
\newcommand{\Dm}{\operatorname{DM}}

\newcommand{\by}{\xrightarrow}

\newcommand{\iso}{\by{\sim}}

\newcommand{\hocolim}{\operatorname{hocolim}}

\newtheorem{thm}{Theorem}
\newtheorem{prop}{Proposition}
\newtheorem{cor}{Corollary}
\newtheorem{lemma}{Lemma}
\theoremstyle{definition}

\theoremstyle{remark}
\newtheorem{rk}{Remark}
\newtheorem{rks}[rk]{Remarks}

\begin{document}
\title{Birational motives and the norm residue isomorphism theorem}
\author{Bruno Kahn}
\address{French-Japanese Laboratory in Mathematics and its Interactions\\ 
IRL2025 CNRS and The University of Tokyo\\
Graduate School of Mathematical Sciences\\
3-8-1 Komaba, Meguro, 153-8914\\
Tokyo, Japan}
\email{bruno.kahn@imj-prg.fr}
\date{February 20, 2025}
\begin{abstract}We point out a relationship between the norm residue isomorphism theorem of Suslin-Voevodsky-Rost and the theory of birational motives, as well as its generalisation to ``higher jets''.
\end{abstract}
\maketitle

Let $k$ be a perfect field, and let $\DM^\eff$ denote Voevodsky's (unbounded) triangulated category of effective Nisnevich motivic complexes over $k$.  For $D\in \DM^\eff$ and $n\ge 0$, we have an adjunction morphism
\begin{equation}\label{eq6}
\uHom(\Z(n),D)(n)\to D
\end{equation}
where $\uHom$ is the internal Hom of $\Dm^\eff$ and, as usual, we abbreviate the notation $\otimes \Z(n)$ to $(n)$. The following lemma is well-known:

\begin{lemma}\label{l3} The morphism \eqref{eq6} is an isomorphism if and only if $D$ is divisible by $\Z(n)$, i.e. if $D$ is of the form $E(n)$.
\end{lemma}

\begin{proof} By Voevodsky's cancellation theorem \cite{voecan}, twisting by $n$ is fully faithful on $\Dm^\eff$.
\end{proof}

Recall that $\Dm^\eff$ carries a ``homotopy'' $t$-structure. Let now $\DM^\eff_\et$ be the triangulated category  of étale motivic complexes. It also carries a homotopy $t$-structure for which the ``change of topology'' functor
\begin{equation}\label{eq0}
\DM^\eff\by{\alpha^*}\DM^\eff_\et
\end{equation}
is $t$-exact; the functor $\alpha^*$ has a right adjoint $R\alpha_*$ \cite[C.4]{kl}. The main result of this note is:

\begin{thm}
\label{t2} Suppose that $D=R\alpha_*C$ with $C\in \Dm_\et^\eff$ bounded and torsion (i.e. that $C\otimes \Q=0$). Then \eqref{eq6} is an isomorphism.
\end{thm}

This provides a large quantity of objects of $\DM^\eff$ which are infinitely divisible by $\Z(1)$, of a quite different nature from those of \cite[Rem. 1.10]{hk}.

\begin{cor}\label{c2} Suppose that $D=R\alpha_*C$ with $C\in \Dm_\et^\eff$ bounded. Then the cone of \eqref{eq6} is uniquely divisible (multiplication by $m$ is an isomorphism for all $m\ne 0$).
\end{cor}

\begin{proof} Consider the exact triangles $C\by{m} C\to C\otimes \Z/m\by{+1}$.
\end{proof}

\begin{rks} a) Theorem \ref{t2} is false for $D$ torsion in general, as the example $D=\Z/l$ shows. Similarly, the torsion hypothesis on $C$ is necessary, as the example $C=\alpha^*\Q$ shows.\\
b) For $C\in \DM^\eff_\et$, an adjunction game provides an isomorphism in $\DM^\eff$
\[R\alpha_* \uHom_\et(\alpha^*\Z(n),C)\iso \uHom(\Z(n),R\alpha_* C) \]
where $\uHom_\et$ is the internal Hom of $\Dm_\et^\eff$. Let $m>0$ be an integer invertible in $k$. The change of coefficients functor
\[\DM_\et^\eff(k)\to \DM_\et^\eff(k,\Z/m)\]
has a right adjoint $i_m$ which induces a natural isomorphism
\[\uHom_\et(C',i_m C)\simeq i_m \uHom^m_\et(C'\otimes \Z/m,C)\]
for any $(C',C)\in \DM_\et^\eff(k)\times \DM_\et^\eff(k,\Z/m)$ where $\uHom^m_\et$ is the internal Hom of $\DM_\et^\eff(k,\Z/m)$. Take $C'=\alpha^*\Z(n)$; then $C'\otimes \Z/m=\mu_m^{\otimes n}$. Thus, if $C\in \DM_\et^\eff$ is bounded and $m$-torsion, the isomorphism \eqref{eq6} of Theorem \ref{t2} for $D=R\alpha_* C$ takes the form  
\[R\alpha_* \uHom^m_\et(\mu_m^{\otimes n},C)(n)\iso R\alpha_* C.
\]
For $C=\mu_m^{\otimes i}$, this gives as a special case an isomorphism
\begin{equation}\label{eq7}
R\alpha_* \mu_m^{\otimes i-n}(n)\iso R\alpha_* \mu_m^{\otimes i}.
\end{equation}
Let $\Gamma_m=Gal(k(\mu_m)/k)$: this is a subgroup of $(\Z/m)^*$. Since $\mu_m^{\otimes n}\simeq \Z/m$ when $n$ is divisible by $\gamma_m=|\Gamma_m|$, this also gives the following corollary.
\end{rks}

\begin{cor}\label{c3} For $C\in \Dm_\et^\eff$ bounded and of exponent $m$, the function $n\mapsto \uHom(\Z(n),R\alpha_*C)$ is periodic of period $\gamma_m$.\qed
\end{cor}

The following reformulation of Theorem \ref{t2} will be useful. Recall that, in \cite{birat-tri}, we introduced and studied a triangulated category of birational motivic complexes  $\DM^\o$; by  \emph{loc. cit.}, Prop. 4.2.5., one has
\[\DM^\o=\DM^\eff/\DM^\eff(1).\]

Higher versions of $\DM^\o$ were introduced in \cite[Def. 3.4]{BKladic} (they are also implicit in \cite{hk}):
\[\DM_{<n}^\eff =\DM^\eff/\DM^\eff(n)\]
so that $\DM_{<1}=\DM^\o$.

By \emph{loc. cit.}, Prop. 3.5,  the localisation functor $\nu_{<n}:\DM^\eff\to \DM_{<n}^\eff$ has a right adjoint $\iota_n$; moreover,  the homotopy $t$-structure of $\DM^\eff$ induces a $t$-structure on $\DM_{<n}^\eff$ via $\iota_n$ (\emph{loc. cit.}, Prop. 3.6). By Lemma \ref{l3}, \emph{Theorem \ref{t2} is then equivalent to saying that $\nu_{<n} R\alpha_*C=0$ for any bounded torsion $C\in \Dm_\et^\eff$.}

\begin{proof}[Proof of Theorem \ref{t2}] 
We start with a Grothendieckian dévissage, first reducing to the case of a single torsion sheaf $\sF$. 
Such $\sF$ comes from the \emph{small} étale site of $\Spec k$ by the Suslin-Voevodsky rigidity theorem \cite{sv1}, i.e. is a Galois module. The functor $\nu_{<n}$ commutes with infinite direct sums as a left adjoint; since étale cohomology of sheaves has the same  property, this reduces us to the case where the stalk(s) of $\sF$ are finite, killed by some $m>0$ that we may further assume to be a power of a prime number $l$ different from the characteristic (since $\Dm_\et^\eff$ is $\Z[1/p]$-linear, where $p$ is the exponential characteristic of $k$ \cite[Prop. 3.3.3 2)]{voetri}). A standard transfer/$l$-Sylow argument now allows us to assume that $\sF$ becomes constant after a finite Galois extension of $k$ whose Galois group $G$ has order a power of $l$. Since the statement is stable under extensions of sheaves and the $l$-group $G$ acts unipotently on $\sF$, we may finally assume that $\sF=\Z/l$. By a standard argument due to Tate ($[k(\mu_l):k]$ is prime to $l$), we further reduce to the case where $\mu_l\subset k$.

We now use Levine's ``inverting the motivic Bott element'' theorem \cite{levine}. The following corrects the presentation in \cite[\S 3]{bkmumbai}. From the isomorphism
\begin{equation}\label{eq5}
\Z(1)\simeq \G_m[-1]
\end{equation}
we get an exact triangle
\begin{equation}\label{eq6.3}
\mu_l[0]\to \Z/l(1)\to \G_m/l[-1]\by{+1}.
\end{equation}

Adding the isomorphism $\Z(n)\otimes \Z/l(1)\iso \Z/l(n+1)$, we get a map in $\DM^\eff$:
\begin{equation}\label{eq2.1}
\Z(n)\otimes \mu_l\to \Z/l(n+1).
\end{equation}

For clarity, write $C\mapsto C\{n\}$ for tensoring with the (constant) sheaf $\mu_l^{\otimes n}$ an object $C\in \DM^\eff$ such that $lC=0$. Thus  $\Z(n)\otimes \mu_l\simeq \Z/l(n)\{1\}$, hence we get from \eqref{eq2.1} another map
\[\Z/l(n)\to \Z/l(n+1)\{-1\}\]
which becomes an isomorphism after sheafifying for the \'etale topology. Iterating, we get a commutative diagram
\[\xymatrix{
\Z/l(0)\ar[r]\ar[drr]&\Z/l(1)\{-1\}\ar[r]\ar[dr]&\Z/l(2)\{-2\}\ar[r]\ar[d]&\ar[dl]\dots\\
&&R\alpha_*\Z/l
}\]
which induces a morphism
\begin{equation}\label{eq2}
\hocolim_r \Z/l(r)\{-r\}\to R\alpha_*\Z/l
\end{equation}
where ``the'' homotopy colimit is the one of Bökstedt-Neeman \cite{bn}; the main theorem of \cite{levine} is that \eqref{eq2} is an isomorphism when $l>2$, or when $l=2$ and either  $\text{char} k>0$ or $-1$ is a square in $k$.

This concludes the proof except for $l=2$ in the exceptional case; we complete this case with  the following proposition, which gives an ``unstable'' version of the previous divisibility at a higher cost. 
\end{proof}

\begin{prop}\label{p1} One has $\nu_{<n} R^q\alpha_*\Z/l=0$ for $q>n$, and $\nu_{<n} R\alpha_*\Z/l=0$. (See comment before the proof of Theorem \ref{t2} for $\nu_{<n}$.)
\end{prop}

\begin{proof}    By the Beilinson-Lichtenbaum conjecture \cite{sv2,voem}, we have an exact triangle for any $q\ge 0$:
\begin{equation}\label{eq1}
\Z/l(q)\{-q\}\to R\alpha_* \Z/l\to \tau_{>q} R\alpha_* \Z/l\by{+1}
\end{equation}

Suppose that $q\ge n$. Applying $\nu_{<n}$ to \eqref{eq1}, we get an isomorphism
\[\nu_{<n}R\alpha_* \Z/l\iso \nu_{<n}\tau_{>q} R\alpha_* \Z/l.\]

Comparing this with the same isomorphism for $q+1$, we get the first statement. Therefore, $\nu_{<n}\tau_{>q} R\alpha_* \Z/l=0$ for any $q\ge n$ (for example for $q=n$) and we conclude. 
\end{proof}

\begin{rks} a) Since $R^0\alpha_*\Z/l=\Z/l$ is a birational sheaf, we have an isomorphism $\Z/l\iso \nu_{<n} R^q\alpha_*\Z/l$ for $q=0$. When $n=1$, this allows us to compute $\nu_{<1} R^1\alpha_*\Z/l$ thanks to the isomorphisms
\[\tau_{\le 1} \nu_{<1} R\alpha_*\Z/l\iso \nu_{<1} R\alpha_*\Z/l=0\]
where the first (\emph{resp.} second) isomorphism follows from the first (\emph{resp.} second) part of Proposition \ref{p1}: this gives
\[\nu_{<1} R^1\alpha_* \Z/l\simeq \Z/l[2].
\]
It is less clear how to compute $\nu_{<n} R^q\alpha_*\Z/l$ for $0<q\le n$ when $n\ge 2$.\\
b) In the spirit of \cite{bkmumbai}, Proposition \ref{p1} for $n=1$ conversely implies formally the Beilinson-Lichtenbaum conjecture: we neglect twists by powers of $\mu_l$ for simplicity, and argue by induction on $q$. Since $\otimes$ is right $t$-exact in $\DM^\eff$ and in view of \eqref{eq5}, there is a natural map
\begin{equation}\label{eq4}
\tau_{\le q-1}R\alpha_* \Z/l\otimes \Z(1)\to \tau_{\le q}R\alpha_* \Z/l
\end{equation}
and we have to show that it is an isomorphism.
The condition $\nu_{<1} R^q\alpha_*\Z/l\allowbreak=0$ means that $R^q\alpha_*\Z/l$ is divisible by $\Z(1)$, which in turn implies that $\tau_{\le q}R\alpha_* \Z/l$ is divisible by $\Z(1)$. But the computation of the étale cohomology of $X\times \G_m$ for smooth $X$ shows that the adjoint of \eqref{eq4}
\[\tau_{\le q-1}R\alpha_* \Z/l\to \uHom(\Z(1),\tau_{\le q}R\alpha_* \Z/l)\]
is an isomorphism, and we conclude with Lemma \ref{l3}.\\
c) If $k$ has virtually finite étale cohomological dimension (\emph{e.g.} is finitely generated), we can relax the hypothesis ``bounded'' to ``bounded below'' in Theorem \ref{t2} and Corollary \ref{c2}. Indeed, we reduce by a transfer argument to the case where $k$ has finite cohomological dimension. As used before, $\nu_{<n}$ commutes with infinite direct sums, and so does  $R\alpha_*$ by Lemma \ref{l2} below. Since, for any $C\in \DM^\eff_\et$, the natural map
\[\hocolim\tau_{\le n} C\to C\]
is an isomorphism, this reduces us to the bounded case.
\end{rks}

\begin{lemma}\label{l2} Suppose that $k$ has finite étale cohomological dimension. Then $R\alpha_*$ commutes with infinite direct sums.
\end{lemma}

\begin{proof} Let $(C_i)_{i\in I}$ be a family of objects of $\DM_\et^\eff$. We want to show that the comparison map in $\DM^\eff$
\[\bigoplus_i R\alpha_* C_i\to R\alpha_*\bigoplus_i C_i\]
is an isomorphism. This can be tested against the generators $M(X)$, where $X$ runs through smooth separated $k$-schemes of finite type. This yields the maps
\[H^n_\Nis(X,\bigoplus_i R\alpha_* C_i)\to H^n_\et(X,\bigoplus_i  C_i),\quad n\in \Z\]
so we are reduced to showing that $H^n_\Nis(X,-)$ and $H^n_\et(X,-)$   commute with $\bigoplus_i$. Since hypercohomology spectral sequences are convergent  (by the hypothesis on $k$ for $H^n_\et(X,-)$), we are reduced to the case of direct sums of sheaves, and the result is true because the Nisnevich and étale sites are both coherent.
\end{proof}

\subsection*{Acknowledgement} I thank Takao Yamazaki for a discussion which led to this article.

\end{document}